\documentclass[11pt]{amsart}
\usepackage{amsmath,amssymb,amsthm}

\usepackage{fourier}
\usepackage[T1]{fontenc}

\usepackage{float}
\usepackage{tikz}
\usetikzlibrary{matrix}

\newcommand{\D}{{\mathbb{D}}}

\newcommand{\R}{{\mathbb{R}}}

\newcommand{\SD}{{\mathcal{D}}}
\newcommand{\SE}{{\mathcal{E}}}
\newcommand{\SW}{{\ker(\SD)}}

\newcommand{\SU}{{\mathcal{U}}}

\newcommand{\SF}{{\mathcal{F}}}

\newcommand{\ST}{{\mathcal{T}}}

\newcommand{\NS}{{\mathbb{S}}}

\newcommand{\Op}{{\mathcal{O}p}}

\newcommand{\e}{\varepsilon}
\newcommand{\dd}{\partial}

\newcommand{\sse}{\subseteq}
\newcommand{\lr}{\longrightarrow}

\newcommand{\Mon}{\operatorname{Mon}}

\newcommand{\std}{{\operatorname{std}}}

\newtheorem{proposition}{Proposition}
\newtheorem{theorem}[proposition]{Theorem}
\newtheorem{definition}[proposition]{Definition}
\newtheorem{lemma}[proposition]{Lemma}

\title{Classification of Engel Knots}

\author{Roger Casals}
\address{Massachusetts Institute of Technology, Department of Mathematics, 77 Massachusetts Avenue Cambridge, MA 02139, USA}
\email{casals@mit.edu}

\author{\'Alvaro del Pino}
\address{Utrecht University, Department of Mathematics, Budapestlaan 6, 3584 Utrecht, The Netherlands}
\email{a.delpinogomez@uu.nl}

\begin{document}

\begin{abstract}
Let $(M,\SD)$ be an Engel 4-fold, we show that the scanning map from the space of Engel knots to the space of formal Engel knots is a weak homotopy equivalence when restricted to the complement of the $\ker\SD$-orbits . This is a relative, parametric and $C^0$-close h-principle.
\end{abstract}
\maketitle

\section{Introduction}
A maximally non--integrable $2$--distribution $\SD$ in a smooth 4-fold $M$ is called an Engel structure. This class constitutes one of the four possible geometries of topologically stable distributions, the other three being contact structures, even--contact structures, and vector fields \cite{ArGi,Ca}. In \cite{CP3} the authors studied the homotopy type of the space of Engel structures on a given smooth 4-fold and showed that there exists an Engel structure in each formal class.
% The upcoming works \cite{CP2,PV} identify classes of Engel structures whose behaviour is strictly algebraic topological, in line with contact geometry \cite{BEM}.
The relative problem of classifying integral submanifolds is also of central interest \cite{AIM}: it was unknown whether the space of Engel knots abides by geometric or, rather, topological constraints. This is the problem we address and solve in this article.

In precise terms, consider the space of Engel knots
$$\SE(M,\SD) = \left\{\gamma\in C^\infty(S^1,M): \gamma \mbox{ embedding}, \gamma'(t) \in \SD_{\gamma(t)}\right\}$$
and its formal analogue, the space of formal Engel knots, where the derivative is decoupled:
$$ \SE^f(M,\SD) = \left\{(\gamma,(F_s)_{s\in[0,1]})\in C^\infty(S^1,M)\times\Mon(TS^1,\gamma^*(TM)):
		\begin{array}{l}
		\gamma \mbox{ embedding}, F_0=d\gamma,\\
		F_1\in\Mon(TS^1,\gamma^*(\SD))
		\end{array} \right\}.$$
The spaces $\SE(M,\SD)$ and $\SE^f(M,\SD)$ are endowed with the Whitney $C^\infty$-topology \cite{EM,GG}. There exists a continuous forgetful map
$$s: \SE(M,\SD) \lr \SE^f(M,\SD),\quad s(\gamma) =(\gamma,d\gamma),$$
which will be referred to as the scanning map \cite{Se}. $\SD$ contains a line field $\SW$ uniquely defined by the condition $[\SW,[\SD,\SD]] \subset [\SD,\SD]$. Its orbits, which in particular are Engel knots, often present rigid behaviour \cite{BH,PP}, so the scanning map $s$ cannot possibly be a weak homotopy equivalence. By discarding the $\SW$--orbits, the $h$--principle can be salvaged:
\begin{theorem}\label{thm:main}
Let $(M,\SD)$ be an Engel 4-manifold and $\SE^g(M,\SD) \subseteq \SE(M,\SD)$ the subspace of Engel knots that are not everywhere tangent to $\SW$. Then, the scanning map 
\[ s:\SE^g(M,\SD) \lr \SE^f(M,\SD) \]
is a weak homotopy equivalence. In addition, the h--principle is $C^0$--close and relative in the parameter and the domain.
\end{theorem}
The proof relies on the h--principle for $\varepsilon$--Engel knots, stated in Proposition \ref{prop:epsilonhprinciple}, which allows us to manipulate families of knots in local projections \cite{Ge,PP}. This is similar in spirit to the classification of loose Legendrian knots \cite{Mu}.  

While the set of all Engel knots is infinite--dimensional, the set of $\SW$--orbits has finite dimension. Even better, for a $C^\infty$--generic Engel structure the set of $\SW$--orbits is discrete. Additionally, $C^\infty$--perturbations of the Engel structure change drastically the homotopy type of the space of $\SW$--orbits \cite{PP}, so our theorem detects all the data persistent under Engel homotopies. In this sense, Theorem \ref{thm:main} essentially classifies all Engel knots. This is in stark contrast with the classification of Legendrian knots \cite{Ch,Gr}, which is still an area of active research.

The set of connected components $\pi_0(\SE^g(\R^4,\SD_\std))$ for the standard Engel structure in $\R^4$ is computed in \cite{Ad,Ge}. In the work \cite{FMP} the fundamental group $\pi_1(\SE^g(\R^4,\SD_\std))$ is computed with methods different from ours. Theorem \ref{thm:main} implies both results. Note also that the homotopy type of the space of formal Engel knots $\SE^f(M,\SD)$ only depends on the formal homotopy class underlying $\SD$ and, as a corollary of Theorem \ref{thm:main}, we obtain that $\SE^g(M,\SD_0)$ and $\SE^g(M,\SD_1)$ are homotopy equivalent if $\SD_0$ and $\SD_1$ are homotopic as formal Engel structures.

\subsection{Acknowledgements} We are grateful to E. Murphy, F. Presas, L.~Traynor and T.~Vogel for insightful discussions. We are also pleased to acknowledge the organizers of the workshop {\it Engel Structures} which took place on April 2017 and to the American Institute of Mathematics for sponsoring it. The problem we solve in this article was one of the questions in the workshop. R.~Casals is supported by the NSF grant DMS-1608018 and a BBVA Research Fellowship, and \'A.~del Pino is supported by the grant NWO Vici Grant No.~639.033.312.

\section{Preliminaries}\label{sec:pre}

Let $M$ be a smooth 4-manifold and $\SD\sse TM$ a 2--distribution, i.e.~a 2--rank subbundle of the tangent bundle. The 2--distribution $\SD$ is said to be Engel if $[\SD,[\SD,\SD]]=TM$, where $[\cdot,\cdot]$ denotes the Lie bracket on vector fields; geometrically, this condition is tantamount to the 2-plane field $\SD$ and the 3-plane field $[\SD,\SD]$ being non-integrable at every point \cite{Fr}. In particular, there exists no embedded surface $\Sigma\sse M$ such that $T\Sigma=\SD$ and the only possible integral submanifolds of $\SD$ can be 1-dimensional integral curves, which exist in abundance \cite{Ca,Ge}.

\subsection{$\varepsilon$--Engel knots} \label{ssec:ehorizontal}

The objects that Theorem \ref{thm:main} classifies are 1-dimensional integral submanifolds of $(M,\SD)$, which belong to the more general set of $\e$-integral curves:

\begin{definition}
Let $(M,\SD)$ be an Engel structure, $g$ a Riemannian metric on $M$, $\e\in\R^+$ and $I$ a 1-dimensional manifold. A curve $\gamma: [0,1] \to M$ is said to be an $\e$--Engel arc if the inequality
$$\measuredangle(\gamma_*TI,\SD)\leq\e$$
holds. In the case where $I=\NS^1$ and $\gamma$ is embedded, $\gamma$ is said to be an $\e$--Engel knot. In addition, a $0$--Engel arc, resp.~knot, i.e.~such that $\gamma_*TI\sse\SD$, is said to be an Engel arc, resp.~knot.
\end{definition}

For a fixed metric $g$ and sufficiently small $\e$, the angle condition $\measuredangle(L,\Pi)\leq\e$ between a line $L$ and a 2--plane $\Pi$ in $\R^4$ implies that the projection $\pi(L)$ of the line onto the plane is an injection and $\measuredangle(L,\pi(L))\leq\e$ as subspaces of $\R^2\cong\langle L,\pi(L)\rangle$. Denote by $\SE^\e(M,\SD)$ the space of $\e$--Engel knots and define its formal counterpart as
\[ \SE^{\e,f}(M,\SD)= \left\{(\gamma,(F_s)_{s\in[0,1]})\in C^\infty(S^1,M)\times\Mon(TS^1,\gamma^*(TM)) \quad|\quad
		\begin{array}{l}
		\gamma \mbox{ embedding},F_0=d\gamma,\\
		\measuredangle((F_1)_*TS^1,\gamma^*(\SD))\leq\e.
		\end{array} \right\}, \]
where both spaces are endowed with the Whitney $C^\infty$-topology. The reason for introducing $\e$-integral submanifolds is that the partial differential relation defining them is open for $\e\in\R^+$, whereas the relation defining Engel knots is not; the openness of the partial differential relation allows us to use the method of convex integration to classify $\e$-Engel knots for $\e\in(0,\pi/2)$:
\begin{proposition} \label{prop:epsilonhprinciple}
The scanning map $s:\SE^\e(M,\SD)\lr\SE^{\e,f}(M,\SD)$ given by $s(\gamma)=(\gamma,d\gamma)$ is a weak homotopy equivalence. In addition, the projection map $\SE^{\e,f}(M,\SD) \lr \SE^f(M,\SD)$ is a homotopy inverse to the inclusion $\SE^f(M,\SD) \lr \SE^{\e,f}(M,\SD)$.
\end{proposition}

\begin{proof}
The fiberwise defining condition for an $\e$-Engel knot is given by conical subsets of the form
$$\{(x,y)\in\R^2:y^2\leq (x\e)^2\}\sse\R^2, $$
whose convex hull is the total fiber, thus proving that the relation is ample; see \cite[Section 19.1]{EM} for details. Since the relation is also open, Gromov's h-principle for ample open relations \cite[Theorem A, Section 2.4.3]{Gr} applies and the statement follows.
\end{proof}

\subsection{Local projections}

Given any point $p\in M$ there exist an open set $\Op(p)\sse M$, an open polydisc $U \subset \R^4(x,y,z,w)$, and a diffeomorphism $\phi:\Op(p)\lr U$ satisfying
\[ \phi_*\SD = \SD_\std = \ker\{dy-zdx\}\cap\ker\{dz-wdx\} = \langle \dd_w, \dd_x+z\dd_y + w\dd_z \rangle. \]
The chart $(U,\phi)$ is referred to as a \emph{Darboux chart} and it can be constructed with the Moser path method \cite{Ca}. Note that the Engel structure $(\R^4,\SD_\std)$ is the tautological distribution on $J^2(\R,\R) \cong \R^4$, where the variable $x\in\R$ is a coordinate in the domain, $y\in\R$ a coordinate on the target, and $z$ and $w$ encode the first and second derivatives of the map, respectively. This naturally extends the isomorphism $(\R^3,\xi_\std)\cong(J^1(\R,\R),\xi_0)$ in the case of 3-dimensional contact structures by further adding the second-order information $j^2(f)$ to the 1-jet $j^1(f)$ of the function $f\in C^\infty(\R)$.

This said, there are two natural contact structures associated to $\SD_\std$ given by the contact forms $dy-zdx$, which reduces the second derivative, and $dz-wdx$, where the original function is itself forgotten. Either of these projections can be used to study Engel structures as contact structures with additional information: J.~Adachi \cite{Ad} works with the sequence of projections
$$J^2(\R,\R)(x,y,z,w)\lr J^1(\R,\R)(x,y,z)\stackrel{\pi_\SF}{\lr} J^0(x,y),$$
and H.~Geiges \cite{Ge} uses the alternative sequence
$$J^2(\R,\R)(x,y,z,w)\stackrel{\pi_G}{\lr} J^1(\R,\R)(x,z,w)\stackrel{\pi_\SF}{\lr} J^0(x,z),$$
where $\pi_\SF$ is the standard front projection in contact geometry \cite{ArGi,Ge0}.

In order to prove Theorem \ref{thm:main} we will use the Geiges projection $\pi_G$. Observe that an Engel knot in $(\R^4,\SD_\std)$ is mapped to an immersed Legendrian in $(\R^3,\xi_\std)$ under the projection $\pi_G$. Such a curve can be manipulated in either the front or the Lagrangian projections, both of which will be used in Theorem \ref{thm:main}. Conversely any Engel arc with domain $I$
\[ \gamma(t) = (x(t),y(t),z(t),w(t)): I \lr J^2(\R,\R) \]
can be recovered from either of the two projections
$$(\pi_\SF\circ\pi_G)(x,y,z,w)=(x,z),\qquad (\pi_L\circ\pi_G)(x,y,z,w)=(x,w).$$
Indeed, in the first case the inclusion $\gamma_*TS^1\in\ker\{dy-zdx\}$ allows us to compute the $y$--coordinate as the integral of the 1-form $zdx$ along a path. Similarly, the vanishing of $dz-wdx$ along $\gamma$ states that the $w$--coordinate is the slope of the projected curve. In the alternate case in which we are given the image of $\pi_L\circ\pi_G$, both $z$ and $y$ are recovered by integration. Figure \ref{fig:EngelKnots} illustrates the projections of some Engel arcs, note that the leftmost projection lifts to a closed Legendrian in $(J^1(\R,\R),\xi_\std))$ but only to an open arc in $(J^2(\R,\R),\SD_\std))$.

\begin{center}
\begin{figure}[h!]
\centering
  \includegraphics[scale=0.85]{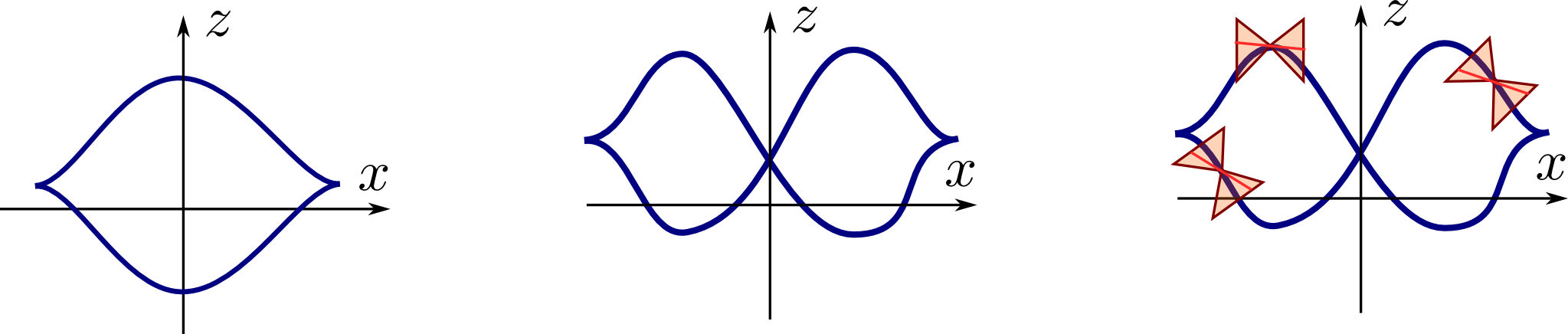}
  \caption{Open Engel arc (Left), a closed Engel knot (Center) and an $\e$-Engel arc (Right).}
  \label{fig:EngelKnots}
\end{figure}
\end{center}

In this line, if we are given the front projection of an $\e$--Engel arc $\gamma$, we are able to recover the missing coordinates up to an error of order $\e$. If the $\e$-error belongs to the derivative, this can be pictorially represented by adding $\e$-cones on top of each point of the front projection, indicating the possible tangent spaces, and a line contained in that $\e$-cone representing the formal slope. Note that for a fixed Darboux chart $(U,\phi)$ with $U \sse M$, there exists $\e\in\R^+$ such that any $\e$--Engel knot intersecting $U$ projects to an immersed $\e$--Legendrian via the projection $\pi_G$ associated to $(U,\SD_\std)$; this $\e$--Legendrian further projects to an immersed planar curve under the Lagrangian projection $\pi_L$.

\subsection{Non--degenerate Engel knots}

Let $(M,\SD)$ be an Engel 4-manifold with kernel $\SW\sse\SD$. In \cite{BH} it is shown that intervals tangent to $\SW$ might possess no deformations relative to their ends as curves tangent to $\SD$. This can be adapted to prove that Engel knots tangent to the line field $\SW$ can form isolated components in the space of Engel knots $\SE(M,\SD)$; these components are geometric and cannot be detected at a formal level, thus the $h$--principle cannot possibly hold for $\SE(M,\SD)$.

Theorem \ref{thm:main} is instead an h-principle for the rest of Engel knots, which constitute the space
$$\SE^g(M,\SD) = \{ \gamma \in \SE(M,\SD):\exists p\in\gamma\mbox{ with }T_p\gamma\cap\SW=\{0\}\}.$$
Its complement $\SE(M,\SD)\setminus\SE^g(M,\SD)$ is actually small, e.g. its Hausdorff dimension is at most 3. Note also \cite{PP} that it is not invariant under Engel deformations and it is a discrete set if $\SD$ is a $C^\infty$--generic Engel structure. As far as Engel \emph{topology}, and not geometry, is concerned, Theorem \ref{thm:main} is sharp.

Engel knots in the space $\SE^g(M,\SD)$ can contain points of tangency with the line field $\SW$. Nevertheless, any subspace in $\SE^g(M,\SD)$ can be assumed to have reasonable tangencies up to $C^\infty$--perturbation:
\begin{lemma}$($\cite[Theorem 1]{PP}$)$\label{lem:genericity}
Let $(M,\SD)$ be an Engel 4-manifold, $K$ a compact parameter space, and $\gamma: K \lr \SE^g(M,\SD)$ a $K$-family of Engel knots. Then there exists a $C^\infty$--perturbation $\tilde\gamma$ of $\gamma$ such that the set 
$$\Sigma=\{(t,k) \in S^1 \times K: \tilde\gamma(k)'(t)\in \ker(\SD)(\tilde\gamma(k)(t))\}$$
is in general position with respect to the projection $S^1 \times K \lr K$.
\end{lemma}
In consequence the set $\Sigma$ admits a Thom--Boardman stratification given by its tangencies with the fibration $S^1 \times K \to K$; we shall refer to a family $\gamma: K \lr \SE^g(M,\SD)$ satisfying the conclusions of the above lemma as a family with \emph{generic tangencies}.

\section{Reduction and Extension}\label{sec:proof}

Theorem \ref{thm:main} is an h-principle, and we shall prove it with the two-step strategy of reduction and extension; this is the approach taken in essentially all existing h-principles in the field \cite{BEM,CP3,Mu} and serves as an indication that the flexible aspects of Engel structures behave in line with the flexible aspects of contact topology, which is also illustrated in Section \ref{app}.

In explicit terms, Theorem \ref{thm:main} states that a family of formal Engel knots
$$\gamma: (\D^k,\partial\D^k) \to (\SE^f(M,\SD),\SE^g(M,\SD))$$
can be homotoped, relative to the holonomic family $\gamma|_{\partial\D^k}$, to a family $\tilde\gamma: \D^k \lr \SE^g(M,\SD)$. Note that by Lemma \ref{lem:genericity} we can and shall assume that the Engel family $\gamma|_{\partial\D^k}$ has generic tangencies with $\ker(\SD)$, which in a local Darboux chart $(\R^4(x,y,z,w),\SD_\std)$ is given by the line field $\partial_w$. Theorem \ref{thm:main} will be proven by homotoping $\gamma$ to an Engel family by using local $C^0$--small perturbations, from which the $C^0$-closeness of the resulting h-principle is immediate. Since the formal curves $\gamma$ are assumed to be embedded, they are guaranteed to remain so as long as these local perturbations do not introduce local intersections.

In our proof of the reduction process, which is the content of Subsection \ref{ssec:reduction}, we will need two technical lemmas; we have isolated these in the following subsection.

\subsection{Two Reduction Lemmas} \label{ssec:auxiliary}

The first technical lemma that we need reads as follows:
\begin{lemma}$($\cite[Proposition 29]{CP3}$)$\label{lem:triangulation}
Let $V$ be a manifold endowed with a line field $\SF$. Then
\begin{enumerate}
\item Any triangulation $(V,\ST)$ admits a subdivision $(V,\ST')$ such that, up to perturbation, every simplex in $\ST'$ is transverse to the given line field $\SF$,
\vspace{0.1cm}
\item There is a covering $\{\SU(\sigma_i)\}_{\sigma_i \in \ST}$ of $V$ satisfying the following properties:
\vspace{0.1cm}
\begin{itemize}
\item[-] There exist flowbox charts $\phi_i: (I \times \D^k,\coprod I \times \{k\}) \to (\SU(\sigma_i),\SF)$,
\vspace{0.1cm}
\item[-] $\phi_i(I \times \{k\})$ is $C^0$-small and $\phi_i(\{t\} \times \D^k)$ is $C^0$--close to each simplex $\sigma_i$,
\vspace{0.1cm}
\item[-] $\SU(\sigma_i) \cup \SU(\sigma_j) \neq \emptyset$ if and only if one is a subsimplex of the other,
\vspace{0.1cm}
\item[-] If $\sigma_i$ is not top dimensional, then $\phi_i(\cup_{\tau \subsetneq \sigma_i} \SU(\tau)) = I \times A_i$, with $A_i \subset \D^k$ some subset.
\end{itemize}
\end{enumerate}
\end{lemma}
% \begin{proof}
% The existence of such a cover is proven in detail in \cite[Proposition 29]{CP3}. Just let us remark that (1) follows by Thurston's jiggling and (2) can be achieved by making $\SU(\sigma)$ be a flowbox neighbourhood of a disc obtained by slightly shrinking $\sigma$. The last property in particular follows by taking the neighbourhoods to be ``tall'' in the $\SF$ direction (in comparison to the other directions in which the simplices are thickened). 
% \end{proof}

For Lemma \ref{lem:triangulation} the reader is encouraged to consider the case $V=S^1 \times K$, $K$ a compact parameter space, and the foliation $\SF = \coprod_{k\in K} S^1 \times \{k\}$. In short, the covering in Lemma \ref{lem:triangulation} is obtained from a triangulation constructed via Thurston's Jiggling Lemma by carefully thickening the simplices.

The second technical lemma is more geometric in nature: it proves that an $\e$--Engel arc can be locally modified to be Engel at a point, and this result holds parametrically and relatively. This is in line with the h-principles for open manifolds where the geometry at the ends is not constrained \cite[Theorem 7.2.2]{EM}.

\begin{lemma} \label{lem:integration}
Let $\gamma_k: I \lr (\R^4,\SD_\std)$ be a $K$-family of $\e$-Engel arcs, with $K$ a compact space, such that $\{\gamma_k\}_{k\in \Op(A)}$ are Engel arcs for a compact subset $A\sse K$. Then, given $(t_k)_{k\in K}\in I$, there exist neighbourhoods $\{\Op(t_k)\}_{k\in K}$ and $\e$--Engel arcs $\{\tilde\gamma_k\}_{k\in K}$ such that:
\begin{itemize}
\item $\tilde\gamma_k(t) = \gamma_k(t)$ if $(t,k) \in I\times\Op(A)\cup (I\setminus\Op(t_k))\times K$,
\item $\tilde\gamma_k$ is an Engel arc in an even smaller neighbourhood of $t_k$, for all $k\in K$.
\end{itemize}
\end{lemma}
\begin{proof}
For each $k\in K$, consider the Lagrangian projection $(\pi_L \circ \pi_G \circ \gamma_k)(t)=(x_k(t),w_k(t))$ and choose an arbitrarily small neighbourhood $[a_k,b_k]$ of $t_k$. In that neighbourhood declare the coordinates $(\tilde z_k(t),\tilde y_k(t))$ to be defined by the definite integration on $(x_k(t),w_k(t))$ with initial value $(z_k(a_k),y_k(a_k))$; by construction this defines an Engel arc $\tilde\gamma_k$ over $[a_k,b_k]$. Let us choose $\delta\in\R^+$ small enough and linearly interpolate back to $\gamma_k$ in the larger neighbourhood $[a_k-\delta,b_k+\delta]$:
$$\tilde\gamma_k(t)=\dfrac{a_k-t}{\delta}\gamma_k(t) + \dfrac{a_k+\delta-t}{\delta}\tilde\gamma_k(t),\quad t\in[a_k-\delta,a_k],$$
$$\tilde\gamma_k(t)=\dfrac{t-b_k}{\delta}\gamma_k(t) + \dfrac{b_k+\delta-t}{\delta}\tilde\gamma_k(t),\quad t\in[b_k,b_k+\delta].$$
Since the $x$ and $w$ coordinates of $\tilde\gamma_k$ and $\gamma_k$ already agree, this amounts to interpolating in the $y$ and $z$ coordinates. This linear interpolation can also be readily made smooth at the ends by inserting the appropriate cut-off functions.

It holds that $\gamma_k=\tilde\gamma_k$ if $k\in \Op(A)$, and the two families coincide on the complement of the neighbourhood $\Op(t_k)=[a_k-\delta,b_k+\delta]$. If $|a_k-b_k|$ and $\delta$ are chosen small enough, the perturbation is $C^0$--small. 
\end{proof}

Note that if the initial $K$-family $\gamma_k$ in Lemma \ref{lem:integration} is embedded then the resulting family $\tilde\gamma_k$ can be assumed to be embedded. This is the form in which the lemma shall be used in the reduction process.

\subsection{Reduction} \label{ssec:reduction}

The first step in the reduction is to apply Proposition \ref{prop:epsilonhprinciple} to the formal Engel $K$-family $\{\gamma_k\}$ with $\e>0$ small but fixed; this allows us to reduce to the case where $\{\gamma_k\}_K\in \SE^\e(M,\SD)$. By applying Lemma \ref{lem:genericity} we perturb the family $\{\gamma_k\}$ to yield a $K$--family of $\e$-Engel knots with generic tangencies over $k \in \Op(\dd K)$, which we still denote by $\{\gamma_k\}$. In order to conclude Theorem \ref{thm:main} it suffices to work with $K=\D^k$, as we shall do onwards.

% It will be convenient to think of this family as a fibered map 
% \[ \gamma(t,k) = \gamma_k(t): \NS^1 \times \D^k \to (M,\SD). \]

% Let us fix a finite covering $\{U_i\}$ of $M$ by Darboux balls. If the constant $\e$ is sufficiently small then each curve $\gamma(k)$ projects to an immersed $\e$--legendrian under the Geiges projection of every $U_i$ it intersects.

\subsubsection{Euclidean Reduction} \label{sssec:triang1}

Let us consider the Engel 4-manifold $(M,\SD)$ and apply Lemma \ref{lem:triangulation} to produce an arbitrarily fine covering $\SU(\sigma_i)$ of $S^1 \times \D^k$. Denote the corresponding charts by $\phi_i: I \times \D^k \lr \SU(\sigma_i)$ and let us work inductively along the skeleton $\{\sigma_i\}$. The $\D^k$--families $\gamma_k\circ\phi_i$ have image in a Darboux ball $U_i$, which has a corresponding Geiges projection $\pi_G$. Up to a $C^\infty$-perturbation, we may suppose that $\pi_L\circ\pi_G\circ\gamma_k\circ\phi_i$ is in general position with respect to the line field $\ker(\SD_\std)=\langle\dd_w\rangle$.

For each simplex $\sigma_i$ below the top-dimension and $k\in \D^k$, we can apply Lemma \ref{lem:integration} to the family $\gamma_k \circ \phi_i$. This effectively modifies the family $\{\gamma_k\}$ over the neighbourhood $\SU(\sigma_i)$, turning the $\gamma_k$ into Engel arcs over possibly smaller domains, which we still denote by $\SU(\sigma_i)$. We also need to apply Lemma \ref{lem:integration} relatively to the boundary of the model, which we do. The relative nature of Lemma \ref{lem:integration} implies that this deformation is relative to the previous steps and, since the covering is arbitrarily fine, Lemma \ref{lem:integration} ensures that the deformation is $C^0$-small. This implies that the embedding property is preserved. This argument hence reduces the set-up to the top-dimensional simplices, each of which has image in a Darboux ball $(M,\SD)=(U \subset \R^4,\SD_\std)$.

% To ensure the generic tangency condition, we may $C^\infty$--perturb $\pi_\lagrangian \circ \pi_\Geiges \circ \gamma \circ \phi_i$ to be in general position with respect to $\partial_w$ before we apply Lemma \ref{lem:integration}.

\subsubsection{Tangency locus} \label{sssec:tangencies}
Consider a $\D^k$-family $\gamma_k: I \lr (\R^4,\SD_\std)$ of $\e$-Engel knots such that for $(t,k) \in \Op(\partial(I\times\D^k))$ the $\{\gamma_k\}$ are Engel knots with generic tangencies. Let us argue that we can reduce to the case where the $\e$-Engel knots are genuine Engel knots near the locus of points where the family $\{\gamma_k\}$ becomes tangent to $\ker(\SD_\std)$.
% In fact, it suffices for $\{\gamma_k\}$ to appear transverse to $\ker\SD_\std$ from the perspective of the projection $\pi_G:\R^4(x,y,z,w)\lr\R^3(x,z,w)$.

Let $\Sigma=\{(t,k): T(\pi_L \circ \pi_G \circ\gamma_k)(t)\cap\langle\dd_w\rangle\neq\{0\}\}$ be the set of tangencies of the Lagrangian projection with $\partial_w$, which is the projection of the kernel $\ker(\SD_\std)$. $\Sigma$ behaves as a smooth manifold, up to $C^\infty$--perturbating $\gamma_k$ relative to the boundary of the model, thanks to the Thom-Boardman stratification. Apply Thurston's Jiggling Lemma to construct a triangulation $\ST_\Sigma$ of $I\times\D^k$ which is transverse to the fibres of $I\times\D^k \lr \D^k$, restricts to a triangulation of $\Sigma$ away from the locus where $\Sigma$ is tangent to the fibers, and is $C^0$--close to a triangulation of $\Sigma$ in that locus. In line with Lemma \ref{lem:triangulation}, $\ST_\Sigma$ induces a covering of $\Sigma$ by flowboxes, since $C^0$--closeness suffices for the construction. Then Lemma \ref{lem:integration} applies again to deform, inductively on the covering, the $\e$-Engel arcs $\gamma_k$ to Engel arcs in a small neighbourhood of $\Sigma$.

In order to preserve the above deformation, triangulate $I\times\D^k$ once more, this time relative to both $\partial(\Op(\Sigma))$ and the boundary of the model. The reduction scheme in Subsection \ref{sssec:triang1} applies again and reduces the initial problem of deforming a $K$-family $\{\gamma_k\}:I\lr(M,\SD)$ of formal Engel knots to the problem of studying a $\D^k$-family of $\e$-Engel knots $\gamma_k: I \lr U \subset (\R^4,\SD_\std)$ satisfying
\begin{itemize}
\item[-] $\{\gamma_k\}$ are Engel knots in the boundary of the model $(t,k) \in \Op(\partial (I \times \D^k))$, 
\item[-] $\pi_\SF \circ \pi_G \circ \gamma_k$ is an embedded interval transverse to $\partial_z$; in particular, it contains no cusps.
\end{itemize}
Note that the second condition is guaranteed by the deformation near the tangency locus $\Sigma$, since the front projection $\pi_\SF:\R^3(x,z,w)\lr\R^2(x,z)$ collapses the Legendrian line field $\langle\dd_w\rangle\sse(\R^3(x,z,w),\xi_\std)$; embeddedness of the projection follows by choosing a fine enough triangulation in the above argument.

In short, Proposition \ref{prop:epsilonhprinciple} and Subsections \ref{sssec:triang1} and \ref{sssec:tangencies} prove that it is possible to work with $\e$-Engel arcs in arbitrarily small Darboux charts that, under the front of the Geiges projection, map to almost horizontal embedded intervals.

\subsection{Extension} \label{ssec:extension}
Let $\gamma_k: I \lr U \subset (\R^4,\SD_\std)$ be a $K$-family of $\e$-Engel knots obtained via the reduction process in Subsection \ref{ssec:reduction}; a crucial fact is that the Darboux ball $U$ might be small and any further deformations we perform must remain inside of it. In addition, these deformations must be $C^0$-small in order for them not introduce intersection points with other families, and should not introduce self--intersections within the local model. Let us explain how to produce these deformations and thus conclude the proof of Theorem \ref{thm:main}.

We will describe the deformation process in the front projection $\pi_\SF\circ\pi_G$, where a given family $\gamma_k(t) = (x_k(t),y_k(t),z_k(t),w_k(t))$ projects to $(x_k(t),z_k(t))$. As discussed in Subsection \ref{sec:pre}, in the region where the curves $\gamma_k$ are Engel the two remaining coordinates can recovered by derivation and integration
$$w_k(t) = \frac{z_k'(t)}{x_k'(t)},\qquad y_k(t) = y_k(t_0) + \int_{t_0}^t z_k(t) dx_k(t).$$
These operations for $\e$-Engel knots only approximate the actual coordinates of $\gamma_k$ with an $\e$-error. Our method first continuously adjusts the slope of the front in order to $C^0$--approximate $w_k$ and then adjusts its integral so as to $C^0$--approximate $y_k$.

\subsubsection{Derivative adjustment.} First, note that a wrinkled embedding \cite{EM2} can approximate any given slope, which for instance proves that any smooth knot can be $C^0$--approximated by a Legendrian knot \cite[Theorem 3.3.1]{Ge0}. In $3$--dimensional contact topology introducing a wrinkle in the front cannot be achieved by a Legendrian isotopy \cite{Ch,Mu}; that is, Legendrian stabilization changes the Legendrian isotopy type. The crucial fact is to realize that in Engel topology it is possible to continuously introduce wrinkles in the region $k \in \Op(\partial\D^k)$ while remaining within the category of Engel knots. This is depicted in Figure \ref{fig:EngelStab}, where the second step $\mathcal{S}$ describes an Engel isotopy; let us provide the details.

\begin{center}
\begin{figure}[h!]
\centering
  \includegraphics[scale=0.7]{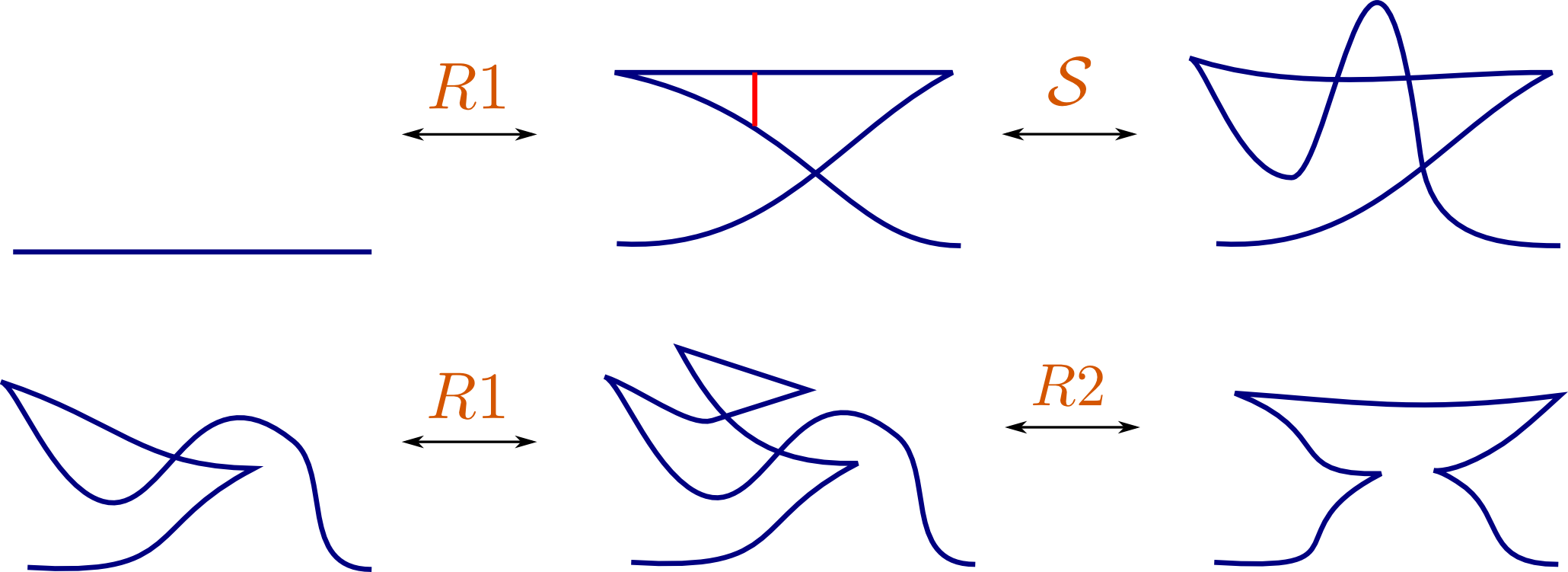}
  \caption{Wrinkling the Engel family $\{\gamma_k\}$.}
  \label{fig:EngelStab}
\end{figure}
\end{center}

Let $N\in\mathbb{N}$ and $t_i = i/N$, $i=1,\dots,N-1$. The given family $\{\gamma_k\}$ of $\e$-Engel knots consists of Engel knots for $k\in\Op(\dd\D^k)$; they are are parametrized by $[0,1]$. For each time $t_i\in[0,1]$ we introduce a Reidemeister I move in such a manner that these Reidemeister I are disjoint from each other and have slope arbitrarily close to that of $\gamma_k$. In fact, since this is a legitimate move for a Legendrian front projection, we can assume that these Reidemeister I appear parametrically in the region $|k|\in[1-\delta,1]$, for a small $\delta\in\R^+$, thus producing a family $\gamma^{(1)}_k$ which coincides with $\gamma_k$ for $|k|\in\Op(\{1\})$ and agrees with $\gamma_k$ with added Reidemeister I moves for $|k|\in\Op(\{1-\delta\})$, as in the second frame in Figure \ref{fig:EngelStab}.

The next step is to add wrinkles to the family $\gamma^{(1)}_k$ and enlarge them to $C^0$-approximate the desired slope $\{w_k\}$; this is achieved by considering the vertical red segment depicted in Figure \ref{fig:EngelStab} and pulling the strand at the lower end of the segment upwards. The Legendrian lift to $(\R^3(x,z,w),\xi_\std)$ of this movie of fronts produces a family of Legendrians, one of which is not embedded. This happens when the lower strand crosses the upper end of the red segment; since the upper and lower strands have identical slope at the intersection point, the Legendrian lift is only immersed. However, in Engel topology we still have the $y_k$ coordinate, which can be used to ensure that the lift of the movie of fronts to $(\R^4(x,y,z,w),\SD_\std)$ is a family of embedded Engel knots.

Indeed, the move $\mathcal{S}$ for Engel knots is admissible, i.e.~ it preserves the Engel knot condition (and in particular embeddedness) when it is applied to our family $\{\gamma_k\}$ for $k \in \Op(\partial(\D^k))$; in addition it can be performed in an arbitrarily small Darboux chart. In order to verify that the Engel condition is preserved observe that the slope of $(x_k(t),z_k(t))$ can be assumed to be $C^0$-close to the original one and we can directly adjust $w_k(t)$ in the given Darboux chart to still be given by derivation; similarly, the change in the integral giving $y_k(t)$ can be made arbitrarily small and the projection can be adjusted slightly to ensure that $y_k$ is modified relative to the ends $t=\{0,1\}$. Regarding embeddedness, the projection $\pi_G \circ \gamma_k$ describes an embedded Legendrian except when the self--tangency appears in the front, but the integral of $zdx$ over the segment connecting the two self--tangency points is necessarily non--zero, which implies that their $y_k$--coordinates are different.

In consequence we can deform the family $\{\gamma^{(1)}_k\}$ to a family $\{\gamma^{(2)}_k\}$ such that
\begin{itemize}
 \item[-] $\{\gamma^{(1)}_k\}=\{\gamma^{(2)}_k\}$ for $|k|\in[1-\delta,1]$.
 \item[-] $\{\gamma^{(2)}_k\}$ coincides with $\{\gamma^{(1)}_k\}$ with the added move $\mathcal{S}$, for $|k|\in[1-2\delta,1-\delta]$.
\end{itemize}
In the second item, the move $\mathcal{S}$ is inserted parametrically: the parameter in the movie coincides with $|k|\in[1-2\delta,1-\delta]$ so that for $|k|\in\Op(\{1-\delta\})$ we have the second front in Figure \ref{fig:EngelStab} and for $|k|\in\Op(\{1-2\delta\})$ the front is given by the third frame.

To conclude this adjustment, choose $\delta\in\R^+$ sufficiently small so that the $\{\gamma_k\}$ are Engel knots for $|k|\in[1-3\delta,1]$. Then, perform the remaining steps depicted in Figure \ref{fig:EngelStab} parametrically in the region $|k|\in[1-3\delta,1-2\delta]$ to obtain a deformed family $\{\gamma^{(3)}_k\}$. Then enlarge the wrinkles in $\{|k|\leq1-3\delta\}$ to $C^0$-approximate the desired slopes; it can be done both parametrically in $k\in\D^k$ and relative to $\{|k|\geq1-3\delta\}$ and $\{t \in \Op(0,1)\}$. This is the only part in which we actually need a large number of wrinkles and thus $N\in\mathbb{N}$ must be chosen initially sufficiently large to ensure that the slope is $C^0$-approximated. Once the wrinkled front projections $C^0$-approximate the required derivatives, we redefine the fourth Engel coordinate as $w_k(t) = \frac{z_k'(t)}{x_k'(t)}$ for all $k\in\D^k$ and denote the resulting family by $\{\gamma^{(4)}_k\}$.

\subsubsection{Area adjustment.} Given the previous subsections, we suppose that $\gamma^{(4)}_k$ are $\e$-Engel knots whose $\pi_G$-projections are legitimate legendrian immersions and differ from being genuine Engel knots due to the behaviour of their $\{y_k\}$ coordinate.

The $\{y_k\}$ should be given by the integral of the 1-form $zdx$. We can insert in the family $\{\gamma_k^{(4)}\}$ a further set of Reidemeister I moves at times $\tilde t_i = (2i+1)/2N$, for $1\leq i\leq N-1$, parametrically in $|k|\in[1-\delta,1]$. These Reidemeister I moves come in two collections: the first collection bounds area $A$ and the second collection area $-A$, with $A\in\R^+$ small. This yields a new family $\{\gamma_k^{(5)}\}$ of $\e$-Engel knots which are still Engel for $|k|\in[1-3\delta,1]$. Each collection is essentially an arbitrarily long sequence of nested Reidemeister I moves; the insertion is depicted in Figure \ref{fig:EngelArea}.

These pairs of nested Reidemeister I moves of cancelling area serve now as area controllers. Note that the nesting condition implies that we can insert arbitrarily many of them and, as long as they are small enough, they do not interact with other regions of the curve. Once these area controllers are available we can adjust the integral $\int z_k(t)dx_k(t)$ as we please and, since $N$ can be taken arbitrarily large, this adjustment is performed in an arbitrarily dense collection of points, so we can $C^0$-approximate $y_k(t)$. The final family $\{\gamma_k^{(6)}\}$ is obtained by deforming the areas of each Reidemeister {\it configuration} in order to $C^0$-approximate the coordinates $\{y_k\}$ as we move in the parameter space $k\in K$, always declaring $\{y_k^{(6)}\}$ to be given by the area integral. The area and slope adjustments introduce no local intersections and they $C^0$--approximate the original family of $\e$--Engel knots: we have indeed produced the desired family of Engel knots and the proof Theorem \ref{thm:main} is complete.\hfill$\square$

\begin{center}
\begin{figure}[h!]
\centering
  \includegraphics[scale=0.7]{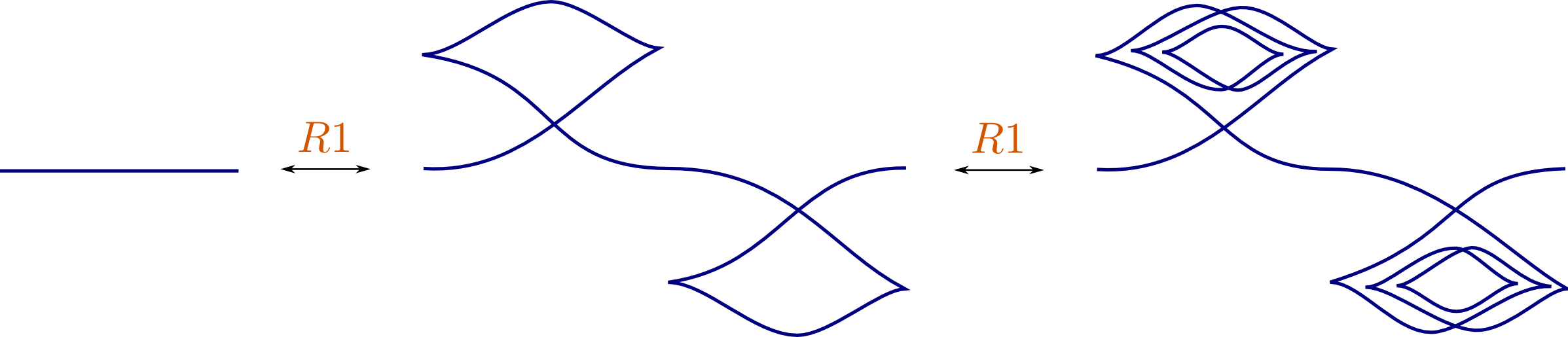}
  \caption{Area controllers modifying the Engel family $\{\gamma_k^{(4)}\}$ to $\{\gamma_k^{(5)}\}$.}
  \label{fig:EngelArea}
\end{figure}
\end{center}

\section{Parametric Fuchs-Tabachnikov Theorem}\label{app}

The methods presented in the proof of Theorem \ref{thm:main} immediately prove the following statement, which is a parametric version of the Fuchs-Tabachnikov Theorem \cite[Theorem 4.4]{FT}; this generalization is natural from the h-principle perspective and it can be deduced from \cite{EM2,Mu}. Its proof is perfectly in line with ours and illustrates the relation between Engel knots and Legendrian knots; let us state it here for completeness:

\begin{theorem}\label{thm:ft}
Let $(Y,\xi)$ be a contact 3-manifold and $\gamma_k: S^1 \lr (Y,\xi)$ a $\D^k$-family of formal Legendrian knots, with $\{\gamma_k\}$ Legendrian knots for $k \in\Op(\partial\D^k)$. Then, there exists a $\D^k$-family $\{\tilde\gamma_k\}$ of Legendrian knots such that $\{\tilde\gamma_k\}$, $k\in\Op(\partial\D^k)$, is obtained from $\{\gamma_k\}$ by stabilization.
\end{theorem}

Exactly as in the proof of Theorem \ref{thm:main} we can first reduce to the case where $\{\gamma_k\}$ are $\e$-Legendrian knots and $(Y,\xi) = (\R^3,\ker\{dz-ydx\})$, and then consider the Lagrangian projection where the $z$-coordinate is recovered as the action $\int ydx$. In the projected immersed curve we introduce small pairs of (nested families of) loops with areas $\pm A$ which serve as area controllers; the crucial difference with the Engel case is that these small pairs of loops cannot be introduced parametrically: they correspond to Legendrian stabilizations, not isotopies. To conclude the argument, we are therefore forced to stabilize the family $\{\gamma_k\}$ for $k \in\Op(\partial\D^k)$ and then the extension argument from Section \ref{sec:proof} above proves Theorem \ref{thm:ft}. \hfill$\square$
%%%%%%%%%%%%%%%%%%%%%%%%%%%%%%%%%%%%%%%%%%%%%%%%%%%%%%%%%%%%%%%%%%%%%%%%%%%%%%%%%%%%%%%%%%%%%%%%%%%%%%%%%%%%%%%%%%%%%%%%%%%%%%%%%%%%%%%
%%%%%%%%%%%%%%%%%%%%%%%%%%%%%%%%%%%%%%%%%%%%%%%%%%%%%%%%%%%%%%%%%%%%%%%%%%%%%%%%%%%%%%%%%%%%%%%%%%%%%%%%%%%%%%%%%%%%%%%%%%%%%%%%%%%%%%%

%%%%%%%%%%%%%%%%%%%%%%%%%%%%%%%%%%%%%%%%%%%%%%%%%%%%%%%%%%%%%%%%%%%%%%%%%%%%%%%%%%%%%%%%%%%%%%%%%%%%%%%%%%%%%%%%%%%%%%%%%%%%%%%%%%%%%%%
%%%%%%%%%%%%%%%%%%%%%%%%%%%%%%%%%%%%%%%%%%%%%%%%%%%%%%%%%%%%%%%%%%%%%%%%%%%%%%%%%%%%%%%%%%%%%%%%%%%%%%%%%%%%%%%%%%%%%%%%%%%%%%%%%%%%%%%
%%%%%%%%%%%%%%%%%%%%%%%%%%%%%%%%%%%%%%%%%%%%%%%%%%%%%%%%%%%%%%%%%%%%%%%%%%%%%%%%%%%%%%%%%%%%%%%%%%%%%%%%%%%%%%%%%%%%%%%%%%%%%%%%%%%%%%%
\end{document}